\documentclass[11pt]{amsart}
\usepackage{geometry}                
\usepackage{graphicx}
\usepackage{amssymb}
\usepackage{epstopdf}
\usepackage{amscd}
\usepackage{amsmath,amscd}
\usepackage{mathtools}
\usepackage{mathrsfs}
\usepackage{comment}

\usepackage[sc]{mathpazo}
\newcommand{\A}{\mathbb{A}}

\newcommand{\oper}{\otimes}

\renewcommand{\P}{\mathbb{P}}

\newcommand{\F}{\mathscr{F}}
\newcommand{\G}{\mathscr{G}}

\newtheorem{thm}{Theorem}[section]

\newtheorem*{theoA}{Theorem A}
\newtheorem*{corB}{Corollary B}
\newtheorem*{theoC}{Theorem C}
\newtheorem*{qst}{Question}

\newtheorem{lm}[thm]{Lemma}

\newtheorem{cor}[thm]{Corollary}
\theoremstyle{definition}
\theoremstyle{definition}\newtheorem{defi}[thm]{Definition}
\theoremstyle{definition}\newtheorem{rmk}[thm]{Remark}

\renewcommand{\epsilon}{\varepsilon}

\newcommand{\nocontentsline}[3]{}
\newcommand{\tocless}[2]{\bgroup\let\addcontentsline=\nocontentsline#1{#2}\egroup}

\title{On rationalizing divisors}
\author{Lorenzo Prelli}
\date{}

\begin{document}
\maketitle
\begin{abstract}
Rational pairs generalize the notion of rational singularities to reduced pairs $(X,D)$. In this paper we deal with the problem of determining whether a normal variety $X$ has a rationalizing divisor, i.e., a reduced divisor $D$ such that $(X, D)$ is a rational pair.

We give a criterion for cones to have a rationalizing divisor, and relate the existence of such a divisor to the locus of rational singularities of a variety.
\end{abstract}

\tableofcontents

\section*{Introduction}
	The development of the minimal model program showed the importance of working with pairs or the form $(X,D)$, where $D$ is a Weil divisor on $X$.
	
	Since many important classes  of singularities of $\mathbb{Q}$-Gorenstiein varieties have a natural extension to pairs, one might ask whether there is a way to generalize the notion of rational singularities to pairs.
	Rational pairs have been recently introduced by Koll\'ar and Kov\'acs \cite[Section 2.5]{kollar2013sing}, and their deformation theory has been explored in Lindsay Erickson's PhD thesis \cite{linthsay}.
	
	In this paper, we focus on the following:
	\begin{qst}
	Given a normal variety $X$, when is there a reduced divisor $D$ such that $(X,D)$ is rational?
	\end{qst}
	We will call such a divisor a \emph{rationalizing divisor}.
	
	The first result is a condition for cones on rational pairs:
	\begin{theoA}[Theorem \ref{conecrit}]
	Let $(X, D)$ be a rational pair, and let $L$ be an ample line bundle. Then the cone $(C(X,L), C(D, L\left|_D\right.  ))$ is a rational pair if and only if
	\[
	H^{i}(X, L^{m}(-D)) = 0\text{ for }i>0, m\geq 0
	\]
	\end{theoA}
	
	As a consequence of the above theorem, we get the following
	\begin{corB}[Corollary \ref{conenec}]
	Let $X$ be a normal variety of dimension $n$ and let $L$ be an ample line bundle on $X$. If the cone $C(X,L)$ has a rationalizing divisor, then
	\[
	H^{n}(X, L^{m}) = 0\text{ for }m\geq 0
	\]
	\end{corB}
	
	Then we give a necessary condition for the existence of a rationalizing divisor:
	\begin{theoC}[\ref{codimrat}]
	Let $(X,D)$ be a rational pair. Then the non rational locus of $X$ has codimension at least $3$.
	\end{theoC}

	We conclude the paper with an example that shows that the above theorems provide necessary but not sufficient conditions to guarantee the existence of a rationalizing divisor.

\section{Background}

	In this section we present the basic definitions and results on rational pairs. We will work on an algebraically closed field $k$ of characteristic zero. The main references are \cite{kollar2013sing}, \cite{definv}.
	\begin{defi}
	A \emph{reduced pair} $(X,D)$ consists of the datum of a normal variety $X$ and a reduced Weil divisor $D$ on $X$.
	\end{defi}
	
	The pair analogue for smoothness is an snc pair.
	\begin{defi}
	A \emph{simple normal crossing} (or snc) pair $(X,D)$ is a pair such that $X$ is smooth, every component $D_{i}$ of $D$ is smooth and the $D_{i}$'s intersect transversely.
	For any reduced pair $(X,D$), its \emph{snc locus} of $(X,D)$ is the largest open set $U$ such that $(U, D\left|_U\right.)$ is snc.
	\end{defi}

	As the reader might expect, we require maps between varieties to respect the pair structure. More precisely, in this paper we will use the following terminology:
	\begin{defi}
	Let $(X,D), (Y,B)$ be reduced pairs. A \emph{birational morphism of pairs} $f$ is a birational morphism $f:Y\to X$ such that $B = f_*^{-1} D$, where $f_*^{-1} D$ denotes the strict transform of $D$.
	\end{defi}
	
	\begin{defi}
	A \emph{resolution} of a pair $(X,D)$ is a birational morphism of pairs $(Y,B)\to (X,D)$, with $(Y,B)$ snc.
	\end{defi}
	
	So far the generalization of smoothness and resolutions was relatively straightforward. In order to define singularities of pairs that behave similarly to rational singularities, we have to restrict to a particular class of resolutions.
	
	\begin{defi}
	If $(X,D)$ is snc, any intersection of the components $D_{i}$ of $D$ is called a \emph{stratum} of $(X,D)$.
	\end{defi}
	
	\begin{defi}
	A resolution $f: (Y,B)\to(X,D)$ is called \emph{thrifty} if it satisfies the following conditions:
	\begin{enumerate}
		\item $f$ is an isomorphism over the generic point of any stratum of $snc(X,D)$
		\item $f$ is an isomorphism over the generic point of any stratum of $(Y,B)$
	\end{enumerate}
	\end{defi}
	The conditions in the above definition are equivalent, respectively, to these two:
	\begin{enumerate}
		\item $f(\text{Ex}f)$ does not contain any stratum of $snc(X,D)$.
		\item $\text{Ex}f$ does not contain any stratum of $(Y,B)$
	\end{enumerate}
	
	Moreover, thrifty resolutions always exist for any pair in characteristic zero \cite{kollar2013sing}.
	
	We can now define a class of resolutions whose properties closely resemble the behavior of the resolution of a variety with rational singularities.
	
	\begin{defi}
	A resolution of pairs $f:(Y,B)\to (X,D)$ is called \emph{rational} if
		\begin{enumerate}
			\item The natural injection $\mathcal{O}_{X}(-D)\hookrightarrow f_{*}\mathcal{O}_{Y}(-B)$ is an isomorphism
			\item $R^{i} f_{*}\mathcal{O}_{Y}(-B) = 0$ for $i > 0$
			\item $R^{i}f_{*}\omega_{Y}(B) = 0$ for $i >0$
		\end{enumerate}
	\end{defi}
	
	It is worth pointing out that over a field of characteristic zero condition (3) is automatically satisfied by an analogue of the Grauert-Riemenschneider vanishing theorem \cite[Prop. 3.6]{definv}.
	Moreover, we will show in the next section that condition (1) always holds if $X$ is normal. This implies that most of the times, in order to prove rationality of a resolution, we will concentrate our efforts in showing that (2) holds.
	
	Finally, we can state our main definition:
	\begin{defi}
	A reduced pair is a \emph{rational pair} if it has a rational thrifty resolution.
	\end{defi}
	
	One can prove that if a reduced pair has a rational thrifty resolution, then every thrifty resolution is rational.
	It is not known whether rational resolutions are necessarily thrifty; it is true for log resolutions, as shown in \cite[Prop. 3.4]{definv}.

\section{An observation on the normality condition}
	
	In the usual definition of rational singularities, one requires that for a resolution of singularities $f:Y\to X$ the natural map $\mathcal{O}_{X}\to f_{*}\mathcal{O}_{Y}$ is an isomorphism: this is equivalent to $X$ being normal.

	The first condition for $(X,D)$ to be a rational pair is that $\mathcal{O}_X(-D) \to f_* \mathcal{O}_Y(-B)$ is an isomorphism for a resolution of pairs $f:(Y,B)\to (X,D)$. We'll show that it's always true if $X$ is normal. More precisely:
	
	\begin{lm}
	Let $X$ be normal and let $D$ be a reduced divisor on $X$. Let $f: (Y,B)\to (X,D)$ be a binational morphism of pairs, with $Y$ normal. Then the natural map $\mathcal{O}_X(-D) \to f_* \mathcal{O}_Y(-B)$ is an isomorphism.
	\end{lm}

	\begin{proof}
	The natural injection $a:\mathcal{O}_X(-D) \to f_* \mathcal{O}_Y(-B)$, induced by the inclusion $\mathcal{O}_X\hookrightarrow f_*\mathcal{O}_Y$, fits into the following commutative diagram with exact rows
	\[
	\begin{CD}
	0 @>>> \mathcal{O}_X(-D) @>>> \mathcal{O}_X @>>> \mathcal{O}_D @>>> 0\\
	@.		@VaVV @V\simeq VV @. \\
	0 @>>> f_*\mathcal{O}_Y(-B) @>>> f_* \mathcal{O}_Y @>>> f_* \mathcal{O}_B 
	\end{CD}
	\]
	and a bit of diagram chasing shows that there's a morphism $b: \mathcal{O}_D \to f_*\mathcal{O}_B $ that fits into the above commutative diagram:
	\[
	\begin{CD}
	0 @>>> \mathcal{O}_X(-D) @>>> \mathcal{O}_X @>>> \mathcal{O}_D @>>> 0\\
	@.		@VaVV @V\simeq VV @VbVV \\
	0 @>>> f_*\mathcal{O}_Y(-B) @>>> f_* \mathcal{O}_Y @>>> f_* \mathcal{O}_B 
	\end{CD}
	\]
	We know that the middle arrow is an isomorphism by normality of $X$. We want to show that $a$ is surjective.
	
	First of all, we claim that the map $b: \mathcal{O}_D\to f_*\mathcal{O}_B$ is injective. Suppose it had a non trivial kernel $K$: this is a sheaf of ideals on $D$. Since $D$ is reduced, then this sheaf of ideals defines a proper closed subscheme $W$ through which $f$ factors: this is impossible since $f$ is dominant.
	
	Now we can take the stalk at a point $x\in X$ and prove surjectivity of $a$ via a diagram chase. Let $\alpha\in f_*\mathcal{O}_Y(-B)_{x}$, and let $\beta$ its image in $f_* \mathcal{O}_{Y,{x}}$. Then this last element corresponds to $\gamma$ via the isomorphism $\mathcal{O}_{X}\simeq f_{*}\mathcal{O}_{Y}$: let $\delta$ be its image in $\mathcal{O}_{D}$. Then $b(\delta)=0$ in $f_{*}\mathcal{O}_{B}$: since $b$ is injective, $\delta = 0$. This implies that $\gamma$ is the image of a (unique) element $\eta\in \mathcal{O}_{X}(-D)$, that maps to $\alpha$. This concludes the proof.
	\end{proof}

\section{A rationality criterion for cones over pairs}

	\begin{defi}[\cite{kollar2013sing}]
	Let $X$ a projective variety and let $L$ be an ample line bundle on $X$. The \emph{affine cone} over $X$ with conormal bundle $L$ is
	\[
	C_a(X, L) = \text{Spec }\bigoplus_{m\geq 0} H^{0}(X, L^{m})
	\]
	\end{defi}
	Note that this is the normalization of the projective cone over $X$ relative to the embedding defined by some power of $L$.

	Let $(X,D)$ be a rational pair, and let $L$ be an ample line bundle on $X$. Let $C = C_a(X, L)$ be the affine cone over $X$, and let $\Delta = C_a(D, L\left|_D\right.)$. In this section, we will find a criterion to find out whether $(C,\Delta)$ is a rational pair, thus generalizing \cite[Prop. 3.13]{kollar2013sing}.

	A suggestion for a necessary condition comes from the following observation:  consider the exact sequence
	\[
	0\to \mathcal{O}_{X}(-D)\to \mathcal{O}_{X}\to \mathcal{O}_{D}\to 0
	\]
	After twisting by $L^{m}$ and taking cohomology we have the long exact sequence
	
	\[
	\begin{CD}
	@.  \ldots @>>> H^{i-1}(D, L^{m}) \to\\
	\to H^{i}(X, L^{m}(-D)) @>>> H^{i}(X, L^{m}) @>>> H^{i}(D, L^{m}) \to \ldots\\
	@.
	\end{CD}
	\]
	
	Assume now $D$ is Cartier and reduced. Then both $X$ and $D$ are Cohen-Macaulay, as $\mathcal{O}_{X}(-D)$ is Cohen-Macaulay (see \cite[Prop. 2.82]{kollar2013sing}).
	
	For the same reason, if $(C,\Delta)$ is a rational pair then $C$ and $\Delta$ are CM, hence the intermediate dimensional cohomology of $L^{m}$ on $X$ and $D$ vanishes (see \cite[Prop. 3.11]{kollar2013sing}). The above long exact sequence implies that
	\[
	H^{i}(X, L^{m}(-D)) = 0
	\]
	at least for $m\geq 0$, $1\leq i \leq \dim X -1$.
	
	As it turns out, this is a necessary and sufficient condition for the cone pair $(C,\Delta)$ to be rational.
	
	\begin{thm}\label{conecrit}
	Let $(X, D)$ be a rational pair, and let $L$ be an ample line bundle. Then the cone $(C(X,L), C(D, L\left|_D\right.))$ is a rational pair if and only if
	\[
	H^{i}(X, L^{m}(-D)) = 0\text{ for }i>0, m\geq 0
	\] 
	\end{thm}
	
	Before the proof, we will need a couple of lemmas.
	The first one is about rational pairs and bundles with a smooth fiber, even though we just state it for projective and affine line bundles. The second one is a simple observation on the composition of derived functors.
	
	\begin{lm}\label{affinebd}
	Let $\pi: Y\to X$ be an $F$-bundle, where $F$ is either $\mathbb{A}^{1}$ or $\P^{1}$ and $X$ is normal. Let $D$ be a reduced divisor on $X$ and let $B=\pi^{-1}D$ be its scheme-theoretic preimage on $Y$. Then $(X,D)$ is a rational pair if and only if $(Y, B)$ is.
	\end{lm}
	
 	\begin{proof}
	Since the question is local on $X$, we can assume $Y = F\times X$ is the trivial bundle. We will start the proof by relating the properties of a resolution of $X$ to the property of a resolution of $Y$
	
	Let $\epsilon: \widetilde{X}\to X$ be a thrifty resolution of $X$: let $\widetilde{Y} = F\times \widetilde{X}$, let $\eta = id \times \epsilon$ and consider the commutative diagram
	\[
	\begin{CD}
		\widetilde{Y} @>\widetilde{\pi}>> \widetilde{X}\\
		 @V\eta VV                                   @V\epsilon VV\\
		Y                    @>\pi>> X 
	\end{CD}
	\]
	Note that the thriftiness of $\epsilon$ clearly implies thriftiness of $\eta$. Moreover, both $\pi$ and $\widetilde{\pi}$ are faithfully flat maps, as they are obtained by base change of the faithfully flat maps $Y\to\text{Spec}k, \widetilde{Y}\to\text{Spec}k,$.
	
	  By \cite[Prop. III.9.3]{hart77}, for any coherent sheaf $\F$ on $\widetilde{X}$  and for any $i\geq 0$ we have a natural isomorphism
	\[
	\pi^{*} R^{i}\epsilon_{*} \F \stackrel{\sim}{\to} R^{i}\eta_{*} \left( \widetilde{\pi}^{*} \F \right)
	\]
	In particular, if $\F = \mathcal{O}_{\widetilde{X}}(-\widetilde{D})$ then one can check that $ \widetilde{\pi}^{*} \F  \simeq  \mathcal{O}_{\widetilde{Y}}(-\widetilde{B})$, so
	\[
	\pi^{*} R^{i}\epsilon_{*} \mathcal{O}_{\widetilde{X}}(-\widetilde{D}) = 0\Leftrightarrow  R^{i}\eta_{*} \mathcal{O}_{\widetilde{Y}}(-\widetilde{B}) = 0
	\]
	Moreover, since $\pi$ is faithfully flat,
	\[
	\pi^{*} R^{i}\epsilon_{*} \mathcal{O}_{\widetilde{X}}(-\widetilde{D}) = 0\Leftrightarrow R^{i}\epsilon_{*} \mathcal{O}_{\widetilde{X}}(-\widetilde{D})  = 0
	\]
	Hence $\epsilon$ is a rational (thrifty) resolution iff $\eta$ is.
	
	Now assume $(X,D)$ is a rational pair: then $\epsilon$ is rational, so $\eta$ is rational and thrifty hence $(Y, B)$ is a rational pair.
	Conversely, if $(Y, B)$ is a rational pair, then the thrifty resolution $\eta$ has to be rational, so $\epsilon$ is a rational thrifty resolution of $(X,D)$.
	
	\end{proof}
	
	\begin{lm}\label{speclemma}
	Let $f:X\to Y$ and $g:Y\to Z$ be morphisms of varieties, and let $\F$ be a coherent sheaf on $X$, and let $\G = f_{*}\G$. If $R^{i}f_{*}\F = 0$ for $i>0$, then
	\[
	R^{i}(g\circ f)_{*}\F = R^{i}g_{*}\G
	\]
	\end{lm}
	\begin{proof}
	The Grothendieck-Serre spectral sequence gives the following identities in the derived category of sheaves over $Z$:
	\[
	\mathbf{R}(g\circ f)_{*}\F = \mathbf{R}g_{*} \circ \mathbf{R}f_{*}\F = \mathbf{R}g_{*} (f_{*}\F) = \mathbf{R}g_{*} \G
	\]
	and taking homology in the $i$th degree yields the desired result.
	\end{proof}

	\begin{proof}[Proof of Theorem \ref{conecrit}]
	The basic idea is the following: blow up the vertex of the normalized cone to get a map
	\[
	p: (Y,B)\to (C,\Delta)
	\]
	where $B$ is the birational transform of $\Delta$. Since locally the blow up is a $\P^1$-bundle over $X$, then $(Y,B)$ is still a rational pair (cfr. Lemma \ref{affinebd}).
	
	Let $g: (\widetilde{Y}, \widetilde{B})\to (Y,B)$ be a rational thrifty resolution of $(Y,B)$, and let $h = f\circ g$.
	
	The strategy for the proof is the following:
	\begin{enumerate}
		\item Show the map $h$ is a thrifty resolution.
		\item Prove that the pair $(C,\Delta)$ is rational if and only if $R^{j}p_{*}\mathcal{O}_{Y}(-B) = 0$ for $j > 0$.
		\item Show that $R^{j}p_{*}\mathcal{O}_{Y}(-B) = \displaystyle{\bigoplus_{m\geq 0} H^{j}(X, L^{m}(-D))}$ for $j>0$
	\end{enumerate}
	We will give two proofs of the third claim: the first one is the most intuitive because it uses the theorem on formal functions, but it requires $D$ to be Cartier. The second one works for Weil divisors, too.
	
	\begin{enumerate}
		\item (\textsc{thriftiness of $h$}) By construction, $h(\text{Ex}(h))$ contains the vertex $P$ of $C$, that is not in the snc locus of $(C,\Delta)$. This means that condition $1$ is satisfied. As for condition $2$, let $\widetilde{E}$ be the strict transform of $E$ in $\widetilde{Y}$. If the intersection of some components $\widetilde{B_j}$ of $\widetilde{B}$ is contained in $\text{Ex}(h)$, then it has to be contained in $\widetilde{E}$ by the thriftiness of $g$. In other words,
	\[
	\widetilde{E} \supseteq \bigcap \widetilde{B}_i
	\]
	and $\bigcap \widetilde{B}_i$ is disjoint from $\text{Ex}(g)$. This implies that
	\[
	E \supseteq \bigcap {B}_i
	\]
	Nevertheless, if two (or more) of the $B_i$'s meet in $E$, since $Y$ is locally a product, then the corresponding components of $\Delta$ meet in $C$. This implies that the intersection of the $B_i$'s cannot be mapped to a point by $p$ (indeed, its image is the cone over the corresponding components of $\Delta$). Therefore $h$ is thrifty.
	
	\item \textsc{Claim:} $(C,\Delta)$ is rational if and only if $R^{j}p_{*}\mathcal{O}_{Y}(-B) = 0$ for $j\geq 0$
	
	Assume first $(C,\Delta)$ is a rational pair: then every thrifty resolution is rational \cite[Cor. 2.86]{kollar2013sing}, so $Rh_* \mathcal{O}_{\widetilde{Y}}(-\widetilde{B}) \sim \mathcal{O}_{C}(-\Delta)$. This implies, by Lemma \ref{speclemma}, that \[Rp_* \mathcal{O}_{{Y}}(-{B}) \sim \mathcal{O}_{C}(-\Delta)\]
	
	Conversely, assume the above condition holds. Then since $h$ is thrifty and $\textrm{char }k=0$, then $R^ih_*\omega_{\widetilde{Y}}(\widetilde{B}) = 0$ for $i>0$, by rational pair analogue of Grauert-Riemenschneider vanishing theorem (see \cite[Prop. 3.6]{definv}). This proves that $h$ is a rational thrifty resolution, hence $(C,\Delta)$ is a rational pair.
	
	\item \textsc{Claim: }$R^{j}p_{*}\mathcal{O}_{Y}(-B) = \bigoplus_{m\geq 0} H^{j}(X, L^{m}(-D))$

	\begin{proof}
	First of all, note that $R^i p_* \mathcal{O}_Y(-B) $ is a skyscraper sheaf supported on the cone point $P\in C(X)$, and its stalk is $H^i(Y, \mathcal{O}_Y(-B))$.
	
	Let $\pi: Y\to X$ the projection: note that it's an affine morphism, therefore
	\[
		H^i(Y, \mathcal{O}_Y(-B)) \simeq H^i(X, \pi_*\mathcal{O}_Y(-B))
	\]
	We will now express $\pi_*\mathcal{O}_Y(-B)$ in terms of $L$.
	
	Consider the exact sequence
	\[
		0\to\pi_*\mathcal{O}_Y(-B) \to \pi_*\mathcal{O}_Y\to\pi_*\mathcal{O}_B\to 0
	\]
	(observe that exactness on the right follows form the fact that $\pi$ is an affine morphism)
	By construction, $\pi_*\mathcal{O}_Y \simeq \bigoplus_{m\geq 0} L^m$, and since $B$ is an affine bundle over $D$, it's easy to check that $\pi_*\mathcal{O}_B\simeq \bigoplus_{m\geq 0} L^m\left|_D\right.$. This implies that $\pi_*\mathcal{O}_Y(-B)\simeq \bigoplus_{m\geq 0} L^m(-D)$: taking cohomology of both sides yields   $H^i(X, \pi_*\mathcal{O}_Y(-B))\simeq \bigoplus_{m\geq 0} H^i(X, L^m(-D))$.
	\end{proof}

	\end{enumerate}

	Combining everything together, the pair $(C,\Delta)$ is rational if and only if $H^i(X, L^m(-D))$ vanishes for every $i>0, m\geq 0$.
	\end{proof}
	
	Theorem \ref{conecrit} has an interesting consequence for the existence of rationalizing divisors on cones.
	
	\begin{cor}\label{conenec}
	Let $X$ be a normal variety of dimension $n$ and let $L$ be an ample line bundle on $X$. If the cone $C(X,L)$ has a rationalizing divisor, then
	\[
	H^{n}(X, L^{m}) = 0\text{ for }m\geq 0
	\]
	\end{cor}

	\begin{proof}
	First of all, observe that there's a natural surjection $\text{Cl }X \to \text{Cl }C(X, L)$, taking a divisor to the cone over it (see \cite[Ex. II.6.3]{hart77}). This means that every divisor on $C(X)$ is the cone over a divisor on $X$, up to linear equivalence.
	
	Now suppose $(C(X), C(D))$ is a rational pair, and let $U$ be $C(X)$ minus the cone point. Since $(U,C(D)\left|_U\right.)$ is a rational pair and $U\to X$ is a $\A^1$-bundle over $X$, it follows from Lemma \ref{affinebd} that $(X,D)$ is a rational pair, too.

	Consider now the exact sequence
	\[
	0\to \mathcal{O}_{X}(L^{m}(-D)) \to \mathcal{O}_{X}(L^{m}) \to \mathcal{O}_{D}(L^{m}) \to 0
	\]
	Then taking cohomology in degree $n$ yields
	\[
	H^{n}(X, L^{m}(-D)) \to H^{n}(X, L^{m}) \to H^{n}(D, \mathcal{O}_{D}(L^{m} ))\to 0
	\]
	The first term vanishes by Proposition \ref{conecrit}, and the rightmost group is zero for dimension reasons. This gives $H^{n}(X, L^{m}) = 0$ for all $m\geq 0$.
	\end{proof}

	\begin{rmk}
		This necessary condition for cones to have a rationalizing divisor implies Erickson's condition on the vanishing of the top-dimensional cohomology \cite[Prop. 3.5.2]{linthsay}.
	\end{rmk}

	\section{Rationalizing divisors and non-rational locus}
	In this section we will give a necessary condition for a variety to have a rationalizing divisor. Many of the statements hold in the context of Du Bois pairs (see \cite[Theorem 1.4.1]{graf2014potentially}).
	
	First of all, we will investigate the case of surfaces. The simplest case is the cone over a curve: let $C$ be a smooth projective curve with an effective divisor $F$ and let $L$ be an ample line bundle on it. If $X$ is the cone over $C$ and $D$ the cone over $F$, then we know by Theorem \ref{conecrit} that if the pair $(X,D)$ is rational then $H^{1}(C, L^{m}(-F)) = 0$ for all $m\geq 0$. It's easy to see that $H^{1}(C, L^{m}(-F))$ surjects onto $H^{1}(C, L^{m})$: using again Theorem \ref{conecrit}, it follows that $X$ has rational singularities.

	This holds in general. Indeed, we have the following	
	\begin{lm}
	 If $(X,D)$ is a rational pair and $X$ is a normal surface, then $X$ has rational singularities.	
	\end{lm}
	
	\begin{proof}
	Since singularities are isolated on normal surfaces, then we can work locally and assume that $X$ has one isolated singularity $x\in X$. Moreover, if $x\notin \text{Supp}D$ then $X$ is rational at $D$ and smooth away from it, hence the claim is trivially true. Therefore we will assume $x\in \text{Supp}D$.
	
	First of all we claim that if $\pi:(Y, B)\to (X, D)$ is a thrifty rational resolution of $X$ with exceptional locus $E$, then $X$ has rational singularities iff
	\[
	R^1\pi_* \mathcal{O}_Y(-E) = 0 
	\]
	
	We will prove this first claim by establishing some vanishing results. First of all let's consider the long exact sequence
	\[
	\begin{CD}
	0@>>> \pi_*\mathcal{O}_Y(-E-B) @>>> \pi_*\mathcal{O}_Y(-B) @>>>\pi_* \mathcal{O}_E(-B\left|_E\right.) @>>>\\
	@>>> R^1\pi_*\mathcal{O}_Y(-E-B) @>>> R^1\pi_* \mathcal{O}_Y(-B) @>>> R^1\pi_*\mathcal{O}_E(-B\left|_E\right.) @>>> 0
	\end{CD}
	\]
	By rationality of the pair $(X, D)$ we have that $R^1\pi_* \mathcal{O}_Y(-B) = 0$, which implies \[R^1\pi_*\mathcal{O}_E(-B\left|_E\right.) = 0\]
	 
	Consider now the exact sequence
	\[
	\begin{CD}
	0@>>> \pi_*\mathcal{O}_E(-B\left|_E\right.) @>>> \pi_*\mathcal{O}_E @>>>\pi_* \mathcal{O}_{B\left|_E\right.} @>>>\\
	@>>> R^1\pi_*\mathcal{O}_E(-B\left|_E\right.) @>>> R^1\pi_*\mathcal{O}_E @>>>R^1\pi_* \mathcal{O}_{B\left|_E\right.}  @>>> 0
	\end{CD}
	\]
	we already know that $R^1\pi_*\mathcal{O}_E(-B\left|_E\right.) = 0$. Moreover, note that since $E$ is exceptional, $\pi: B\left|_E\right. \to x$ has a zero-dimensional image, hence $R^1\pi_* \mathcal{O}_{B\left|_E\right.} = 0$. This implies that $R^1\pi_*\mathcal{O}_E = 0$.
	
	Finally, consider the exact sequence 
	\begin{equation}\label{exseq}
		0\to \mathcal{O}_Y(-E) \to \mathcal{O}_Y \to \mathcal{O}_E \to 0
	\end{equation}
	and apply the functor $\pi_*$ to get
	\[
	\begin{CD}
	0@>>> \pi_*\mathcal{O}_Y(-E) @>>> \mathcal{O}_X @>>>\pi_* \mathcal{O}_E @>>>\\
	@>>> R^1\pi_*\mathcal{O}_Y(-E) @>>> R^1\pi_* \mathcal{O}_Y @>>> R^1\pi_*\mathcal{O}_E @>>> 0
	\end{CD}
	\]
	The exceptional locus $E$ is connected by Zariski's main theorem, hence the sheaf $\pi_*\mathcal{O}_E$ is a skyscraper sheaf supported on $x\in X$ and therefore the map $\mathcal{O}_X \to \pi_* \mathcal{O}_E $ is a surjection. This yields the exact sequence
	\[
	\begin{CD}
	0@>>> R^1\pi_*\mathcal{O}_Y(-E) @>>> R^1\pi_* \mathcal{O}_Y @>>> R^1\pi_*\mathcal{O}_E @>>> 0
	\end{CD}
	\]
	By the vanishing results above, the last sheaf is zero, hence $R^1\pi_*\mathcal{O}_Y(-E) \simeq R^1\pi_* \mathcal{O}_Y$. It follows that $X$ has rational singularities iff $R^1\pi_* \mathcal{O}_Y=0$ iff $R^1\pi_*\mathcal{O}_Y(-E) =0$.
	
	We now claim that \[R^1\pi_* \mathcal{O}_Y(-E-B) = 0\]
	Indeed, twisting the sequence (\ref{exseq}) by $-B$ and taking $\pi_{*}$ we get an exact sequence
	\[
	\begin{CD}
	@. 					 @. \pi_{*}\mathcal{O}_{Y}(-B) @>\alpha>>\pi_{*}\mathcal{O}_{E}(-B\left|_E\right.) @>>>\\
	@>>> R^{1}\pi_{*}\mathcal{O}_{Y}(-E-B)  @>>> 0 @.
	\end{CD}
	\]
	It suffices to show that the map $\alpha$ is surjective, and we can do it by checking surjectivity at the stalk at $x$. Clearly $\mathcal{O}_{E}(-B) \subseteq \mathcal{O}_{E}$, thus $(\pi_{*}\mathcal{O}_{E}(-B))_{x} \subseteq (\pi_{*}\mathcal{O}_{E})_{x}$. We already observed that $\pi_{*}\mathcal{O}_{E} \simeq k_{x}$, so $(\pi_{*}\mathcal{O}_{E}(-B))_{x}$ is either $0$ or $k$. Since we have
	\[
	(\pi_{*}\mathcal{O}_{E}(-B))_{x} = \Gamma(E, \mathcal{O}_{E}(-B\left|_E\right.))
	\]
	 and given that $x\in D$ ensures that $ \mathcal{O}_{E}(-B\left|_E\right.)$ is a proper ideal of $\mathcal{O}_{E}$, we can conclude that $\mathcal{O}_{E}(-B\left|_E\right.) = 0$.
	 
	 Finally, we will be done if we can show that $R^1\pi_* \mathcal{O}_Y(-E-B) = 0$ implies $R^1\pi_* \mathcal{O}_Y(-E) = 0$. This follows easily by considering
	\[
	0\to \mathcal{O}_Y(-B-E) \to \mathcal{O}_Y(-E) \to \mathcal{O}_B(-E) \to 0
	\]
	which yields an exact sequence
	\[
	R^{1}\pi_{*}\mathcal{O}_Y(-B-E) \to R^{1}\pi_{*}\mathcal{O}_Y(-E) \to R^{1}\pi_{*}\mathcal{O}_B(-E)
	\]
	Since $\pi\left|_B\right.$ is finite, then $R^{1}\pi_{*}\mathcal{O}_B(-E) = 0$, and this concludes the proof.

	\end{proof}

	Now we are ready to state the main result of the section:
	\begin{thm}\label{codimrat}
	Let $(X,D)$ be a rational pair. Then the non-rational locus of $X$ has codimension at least $3$.
	\end{thm}

	Before proving the proposition, we have to establish a deformation result for rational singularities. 
	\begin{lm}\label{deflemma}
	Let $D$ be a Cartier divisor of $X$ that has rational singularities. Then on an open neighborhood of $D$ the ambient space $X$ has rational singularities.
	\end{lm}

	\begin{proof}
	Since the question is local, after possibly shrinking we can think of $D$ as being the zero set of a regular function $f$. Consider the map
	\[
	\varphi: X \to \A^1, x\mapsto f(x)
	\]
	We have that $\phi$ is flat, as it's dominant onto a one dimensional integral regular scheme and $X$ is reduced \cite[Prop. III.9.7]{hart77}.
	The special fiber $X_{0}\simeq D$ has rational singularities, hence by Elkik's theorem \cite[Th\'eor\`eme 2]{elkik1978singularites} $X$ has rational singularities in a neighborhood of $D$.
	\end{proof}
	
	Now we can proceed with the proof of the proposition.
	
	\begin{proof}
	We will prove the claim by showing that the non rational locus does not intersect a general complete intersection surface in $X$.
	
	Let $S$ be the intersection of $(\dim X - 2)$ general hyperplanes in $X$. Then $X$ is normal (see for example \cite{seidenberg1950hyperplane}), and by using the Bertini property for rational pairs \cite[Theorem 4.4]{definv}, the pair $(S, D\left|_S\right.)$ is rational. It follows from Lemma \ref{deflemma} that $S$ has rational singularities, so by repeatedly applying the deformation lemma \ref{deflemma} we conclude that $X$ has rational singularities near $S$.
	\end{proof}
	
	The proposition implies, for example, that a cone over a normal surface with a non rational isolated singularity cannot be a part of a rational pair, as the non rational locus is a line, which has codimension $2$.
	
	\section{An interesting example}
	
	In this section we will show that the necessary conditions in the previous sections (Propositions \ref{conenec}, \ref{codimrat})  are necessary for the existence of a rationalizing divisor, but not sufficient. We will construct a threefold with an isolated non-rational singularity at a point, that does not have a rationalizing divisor.
	
	We will achieve this by considering a cone $Y$ over a polarized surface $(X,L)$ with zero top-dimensional cohomology of $L^m$ for $m\geq 0$: to be consisted with the previous sections, $L$ will denote an ample line bundle.
	
	Let $C$ be a smooth curve of genus at least 2 and let $X$ be the ruled surface $C \times \P^{1}$. Let $\pi:X\to C, \xi:X\to\P^{1}$ be the projections. Then $\pi_*\mathcal{O}_X=\mathcal{O}_{C}\oplus\mathcal{O}_{C}$ (see for example \cite[Example V.2.11]{hart77}).
	Let $P\in\P^{1}$ and $Q\in C$ be arbitrary points, and let $C_{0} = \xi^{*}P$ be the canonical section, $f = \pi^{*}Q$ the fiber.
	
	The effective divisor $L:=\mathcal{O}_{X}(C_{0}+f)$ is ample by \cite[Prop. V.2.20]{hart77}.
	Note also that by \cite[Prop. V.2.8]{hart77} we have
	\[
	L = \mathcal{O}_{X}(C_{0})\oper \mathcal{O}_{X}(f) = \mathcal{O}_{X}(1)\oper \pi^{*}\mathcal{O}_{C}(Q)
	\]
	
	It follows that
	\[
	\pi_{*} L = \pi_{*}(\mathcal{O}_{X}(1)\oper \pi^{*}\mathcal{O}_{X}(Q)) 
	\]
	By \cite[Lemma V.2.4]{hart77} , since $D.f=1\geq 0$ then
	\[
	H^{1}(X,L) = H^{1}(C,\pi_{*}L) = H^{1}(C,\mathcal{O}_{X}(Q))\oplus H^{1}(C,\mathcal{O}_{X}(Q))\neq 0
	\]
	where the last claim follows from Riemann-Roch and the assumption on the genus of $C$.
	
	Since $H^{1}(X,L)\neq 0$, then we can deduce that the cone $Y = C(X,L)$ has an isolated, non-rational singularity by applying Proposition \ref{conecrit} with $D=0$. Note that by construction $Y$ is normal.
	
	As for the cohomology of $X$, notice that by Lemma 2.4
	\[
	H^2(X, \mathcal{O}_X) \simeq H^2(C, \pi_*\mathcal{O}_X) = 0
	\]
	where the last equality holds for dimension reasons.
	
	We will now prove that $Y$ does not have any rationalizing divisors. We have to prove that there is no effective (Weil) divisor $D = aC_0+bf$ on $X$ such that $(Y, C(D, L\left|_D\right.))$ is a rational pair, and by the criterion in section one this amounts to showing that there are no integers $a,b$ such that
	\[
	H^1(X, L^m( - D)) = 0 \text{ for all }m\geq 0
	\]
	First note that since $D$ is effective, by taking the intersection with $C_0$ and $f$ necessarily we have $a\geq 0, b\geq 0$. Moreover, by K\"unneth formula we have
	\begin{eqnarray}\label{kunny}
		H^1(X, L^m( - D)) & = &  H^1(X, \mathcal{O}_{X}((m-a)C_0 + (m-b)f))\\
					 & = & H^{1}(X, \xi^{*}\mathcal{O}_{\P^{1}}((m-a)P) \oper \pi^{*}\mathcal{O}_{X}((m-b)Q)) \nonumber \\
					 & = & H^0(\P^1, \mathcal{O}_{\P^{1}}((m-a)P)) \oper H^1 (C, \mathcal{O}_{C}((m-b)Q)) \oplus \nonumber \\
					 & & H^1(\P^1, \mathcal{O}_{\P^{1}}((m-a)P)) \oper H^0(C, \mathcal{O}_{C}((m-b)Q)) \nonumber
	\end{eqnarray}
	We claim that for any choice of $a, b$ there is an $m$ such that $H^1(X, L^m (- D))\neq 0$: this will prove that there is no rationalizing divisor for $Y$. The basic idea is to leverage the fact that $H^1(C, \mathcal{O}_{C}(Q))\neq 0$, which follows from Riemann-Roch.
	
	Assume first $a \neq b+2$. Then we set $m-b = 1$, that is $m = b+1$. We get that
	\begin{eqnarray*}
		H^0(\P^1, \mathcal{O}_{\P^{1}}((b-a+1)P)) \oper H^1 (C, \mathcal{O}_{C}(Q)) \neq 0 &\text{ if }& a\leq b+1\\
		H^1(\P^1, \mathcal{O}_{\P^{1}}((b+1-a)P)) \oper H^0(C, \mathcal{O}_{C}(Q))\simeq  & & \\
		H^0(\P^1, \mathcal{O}_{\P^{1}}((a-b-3)P)) \oper H^0(C, \mathcal{O}_{C}(Q)) \neq 0 &\text{ if }& a \geq b+3
	\end{eqnarray*}
	Since $a\neq b+2$, then at least one of the two tensor products is non zero.
	
	Assume now $a = b+2$. Then the second term in the direct sum of (\ref{kunny}) is then
	\begin{eqnarray*}
	H^1(\P^1, \mathcal{O}_{\P^{1}}((m-b-2)P)) \oper H^0(C, \mathcal{O}_{C}((m-b)Q)) &\simeq \\
	H^0(\P^1, \mathcal{O}_{\P^{1}}((b-m)P)) \oper H^0(C, \mathcal{O}_{C}((m-b)Q))	
	\end{eqnarray*}

	Setting $m = b$ yields $H^0(\P^1, \mathcal{O}_{\P^1}) \oper H^0(C, \mathcal{O}_C)\neq 0$.
	
	This shows that the threefold $Y$ has a non rational singularity at the cone point, yet does not have a rationalizing divisor.

	\subsection*{Acknowledgments}
	The author would like to thank his advisor S\'andor Kov\'acs for his supervision, direction, and many insights into the subject of this project. The author would also like to thank Siddharth Mathur for numerous useful conversations.

	\providecommand{\bysame}{\leavevmode\hbox to3em{\hrulefill}\thinspace}
\providecommand{\MR}{\relax\ifhmode\unskip\space\fi MR }
\providecommand{\MRhref}[2]{%
  \href{http://www.ams.org/mathscinet-getitem?mr=#1}{#2}
}
\providecommand{\href}[2]{#2}

\end{document}